\documentclass[11pt,center]{amsart}
\usepackage{amsmath,amssymb,hyperref}

\usepackage{tikz}
\usetikzlibrary{cd,arrows.meta,calc,decorations.markings,patterns}

\usepackage[bb=ams,scr=euler]{mathalpha}
\usepackage{enumitem}
\setenumerate{label=(\arabic*)}

\usepackage{thmtools}
\declaretheorem[name=Theorem]{theorem}
\declaretheorem[name=Lemma,sibling=theorem]{biglemma}
\declaretheorem[name=Definition,sibling=theorem]{bigdefinition}
\declaretheorem[name=Proposition,sibling=theorem]{bigproposition}
\declaretheorem[name=Corollary,sibling=theorem]{corollary}
\declaretheorem[name=Question,sibling=theorem]{question}
\declaretheorem[name=Conjecture,sibling=theorem]{conjecture}

\declaretheorem[name=Lemma,numberwithin=section]{lemma}
\declaretheorem[name=Claim,sibling=lemma]{claim}

\declaretheorem[name=Remark,sibling=lemma]{remark}
\declaretheorem[name=Definition,sibling=lemma]{definition}
\declaretheorem[name=Example,sibling=lemma]{example}

\makeatletter
\newcommand{\abs}[1]{|#1|}
\newcommand{\bd}{\partial}

\renewcommand{\d}{\mathit{d}}
\newcommand{\id}{\mathit{id}}

\newcommand{\R}{\mathbb{R}}

\newcommand{\set}[1]{\left\{#1\right\}}
\newcommand{\Z}{\mathbb{Z}}
\newcommand{\Q}{\mathbb{Q}}
\newcommand{\C}{\mathbb{C}}
\renewcommand\section{\@startsection{section}{1}{0pt}{-3.5ex \@plus -1ex \@minus -.2ex}{2.3ex \@plus.2ex}{\centering\bfseries}}
\renewcommand{\subsection}{\@startsection{subsection}{2}\z@{.5\linespacing\@plus.7\linespacing}{-.5em}{\normalfont\bfseries}}
\makeatother

\date{\today}
\author{Dylan Cant}
\email{dylan@dylancant.ca}
\address{Département de mathématiques et de statistique, Université de Montréal, Pavillon André-Aisenstadt, 2920 Chemin de la Tour, Montréal, Québec, 
Canada}

\begin{document}
\title[On the strong Arnol'd chord conjecture for domains in $\R^{4}$]{The strong Arnol'd chord conjecture for the boundary of a uniformly convex domain in $\R^{4}$}
\begin{abstract}
  Following the idea of Jungsoo Kang and Jun Zhang, we prove the strong Arnol'd chord conjecture for the boundary of a uniformly convex domain in $\R^{4}$, using an ellipsoid embedding construction due to Oliver Edtmair. We prove a general structural result for Legendrians $L$ which are \emph{eventually equivariantly essential} (E3), in the sense that the $k$th Gutt--Hutchings capacity $c_{k}(D^{*}TL)$ is infinite for $k$ large enough. We show that any E3 Legendrian in the boundary of a Liouville domain $\Omega$ bounds a chord of length at most $\liminf c_{k}(\Omega)/k$.
\end{abstract}
\maketitle

\section{Introduction}
\label{sec:introduction}

The \emph{strong Arnol'd conjecture} is said to hold for a Liouville domain $(\Omega,\lambda)$ provided each Legendrian $L\subset \partial \Omega$ bounds a Reeb chord with distinct endpoints. Here we recall the characteristic foliation of $\partial \Omega$ is spanned by a vector field $R$ satisfying $\lambda(R)=1$, and this vector field is the Reeb flow we consider throughout this text. In this paper we prove the following:
\begin{theorem}\label{theorem:main-R4}
  If $\Omega\subset \R^{4}$ is a uniformly convex domain, then every Legendrian $L\subset \partial \Omega$ bounds a Reeb chord whose length is strictly smaller than the minimal positive period of a Reeb orbit. In particular this chord must have distinct endpoints, and so $\Omega$ solves the strong Arnol'd chord conjecture.
\end{theorem}

In \cite{kang-jtopol-2026} this result is proved, in $\R^{2n}$ for $2n\ge 4$, under a non-degeneracy assumption on the Reeb flow (the assumption is that there is a Morse-Bott component in the space of systoles). Our proof has a different flavour, and we obtain the result (only in $\R^{4}$) without any assumptions on $\Omega$ besides its uniform convexity.

The engine behind our upper bound is the sequence of \emph{Gutt--Hutchings capacities} $c_{k}(\Omega)$ defined in \cite{gutt-hutchings-AGT-2018} (see Definition \ref{definition:GH}). The paper \cite{kang-zhang-arXiv-2024} already uses the relationship between Gutt--Hutchings capacities and the Lagrangian capacity of \cite{cieliebak-mohnke-inventiones-2018} established in \cite{pereira-thesis-2022,pereira-JSG-2025} to prove some cases of the strong Arnol'd chord conjecture.\footnote{We also refer the reader to \cite{viterbo-CRASP-1990} for earlier work relating the Ekeland--Hofer capacities with a version of the Lagrangian capacity.} This paper is a mild variation of the idea in \cite{kang-zhang-arXiv-2024}. Our approach differs in its execution in a few places. Firstly, it does not rely on the results of \cite{pereira-thesis-2022,pereira-JSG-2025}, and, secondly, it does not use the Lagrangian capacity; rather, it uses the idea of \cite[\S2.1]{brocic-cant-arXiv-2026} based on Viterbo restriction and exact embeddings $\iota:D_{\epsilon}T^{*}L\times D(a)\to \Omega$.

\begin{remark}
  While this note was being prepared, \cite{faisal-li-yin-2026} was posted, and it proves similar results to those of \cite{pereira-thesis-2022,pereira-JSG-2025}. The paper \cite{faisal-li-yin-2026} uses the ``de Rham string topology'' approach of \cite{fukaya-NATO-2006,irie-jtopol-2020,li-yin-arXiv-2023,li-shuhao-arXiv-2026}, and bounds the minimal area of $J$-holomorphic disk on a closed spin aspherical Lagrangian in terms of the Gutt--Hutchings capacities.

  Unlike \cite{faisal-li-yin-2026}, the results we state below do not involve any a priori spin or orientability hypotheses on the Legendrians.
\end{remark}

To state our next result, we introduce the following class of manifolds:
\begin{bigdefinition}
  A compact manifold $L$ is said to be \emph{eventually equivariantly essential (E3)} provided that $c_{k}(DT^{*}L)=\infty$ for $k$ sufficiently large.
\end{bigdefinition}
\begin{biglemma}\label{lemma:aspherical}
  All compact aspherical manifolds are E3.
\end{biglemma}
\begin{remark}\label{remark:essential}
  \cite[Corollary 1.1.6, Lemma 4.2.5]{zhao-thesis-2016} and Definition \ref{definition:regular-borel-data} imply that spin E3 manifolds are rationally essential.
\end{remark}

\begin{theorem}\label{theorem:main-tech}
  Let $\Omega$ be a Liouville domain. If $L\subset \partial \Omega$ is an E3 Legendrian submanifold, then $L$ bounds a Reeb chord whose length is at most the limit infimum of $c_{k}(\Omega)/k$ as $k\to\infty$, provided this limit infimum is finite.
\end{theorem}

If $\Omega$ is $4$-dimensional, then $L$ is $1$-dimensional, so $L$ must be a disjoint union of circles, and hence $L$ is E3 by Lemma \ref{lemma:aspherical}. This is what enables us to prove an unconditional statement in Theorem \ref{theorem:main-R4}.

In higher dimensions, there are examples of Legendrians which fail the conclusion of Theorem \ref{theorem:main-tech}, and thus are not E3. For instance:
\begin{example}\label{example:CPn-1}
  Any Lagrangian in $\bar{L}\subset \C P^{n-1}=\partial B(1)/S^{1}$ whose rationality constant $\rho(\bar{L})$ equals $1/d$ admits a $d$-fold Legendrian cover $L\subset \partial B(1)$ whose minimal chord length is $1/d$. In particular if $d<n$, this lift $L$ is not E3.
\end{example}
This shows $S^{n-1}$ is not E3, although this also follows from Remark \ref{remark:essential}. The example also implies:
\begin{corollary}[{\cite[Theorem 1.1]{cieliebak-mohnke-inventiones-2018}}]
  Any aspherical Lagrangian $\bar{L}$ in $\C P^{n-1}$ has rationality constant $\rho(\bar{L})$ at most $1/n$.
\end{corollary}
\begin{proof}
  If $\rho(\bar{L})\ne 0$, then $\rho(\bar{L})=1/d$ for some $d$, since $\bar{L}$ bounds some disk with symplectic area $1$. Example \ref{example:CPn-1} applies, and we conclude a $d$-fold Legendrian lift $L\to \bar{L}$. Via Lemma \ref{lemma:aspherical} we conclude that $L$ is E3. Theorem \ref{theorem:main-tech} applies to $L\subset \partial B(1)$, implying $d\ge n$, as desired.
\end{proof}

Returning to the statement of Theorem \ref{theorem:main-R4}, we can be a bit more precise about the ratio of the shortest chord length divided by length of the systole. We digress for a moment and explain this. The key input is the following consequence of the results of \cite{edtmair-gafa-2024,abbondandolo-edtmair-kang-arXiv-2024,florio-hryniewicz-ASENS-2025}:
\begin{biglemma}\label{lemma:ellipsoid-flex}
  If $\Omega$ is a uniformly convex domain in $\R^{4}$, then there exists a symplectomorphism $\phi$ of $\R^{4}$ such that $\phi(\Omega)$ is contained in an ellipsoid $E(c,cN)$ where $c$ is the length of a systole of $\Omega$ and $N>0$ is a large (but finite) number.
\end{biglemma}
The statement is not literally in the cited papers, but it follows from the construction of \cite{edtmair-gafa-2024}, as we will explain in \S\ref{sec:proof-flex}.

The relevance of this result is that the Gutt--Hutchings capacities for an ellipsoid are well-known:
\begin{bigproposition}[{\cite[Example 1.8]{gutt-hutchings-AGT-2018} and \cite[\S4]{cieliebak-hofer-latschev-schlenk}}]\label{prop:GH-ellipsoid}
  The $k$th Gutt--Hutchings capacity of $E(a_{1},\dots,a_{n})$ is equal to the $k$th largest number in the list: $$a_{1},\dots,a_{n},2a_{1},\dots,2a_{n},3a_{1},\dots,3a_{n},\dots$$ with repetitions. In particular:
  \begin{equation*}
    \liminf_{k\to\infty} \frac{c_{k}(E(a_{1},\dots,a_{n}))}{k}=\frac{1}{1/a_{1}+\dots+1/a_{n}}.
  \end{equation*}
  For $E(c,cN)$, this simplifies to $cN/(1+N)$. \hfill$\square$
\end{bigproposition}

It then holds that $L\subset \partial \Omega$ bounds a chord of length at most $cN/(1+N)<c$, if $\Omega\subset E(c,cN)$ (as in Lemma \ref{lemma:ellipsoid-flex}), as we have:
\begin{equation*}
  \frac{c_{k}(\Omega)}{k}\le \frac{c_{k}(E(c,cN))}{k}.
\end{equation*}
This proves Theorem \ref{theorem:main-R4} modulo Lemma \ref{lemma:aspherical}, Theorem \ref{theorem:main-tech}, and Lemma \ref{lemma:ellipsoid-flex}.

To conclude this section, we note that the hypothesis of uniform convexity cannot be relaxed to convexity.
\begin{example}
  The polydisk $\Omega=D(1)\times D(2)$ contains the Legendrian:
  \begin{equation*}
    L=\set{(z,\bar{z}):\abs{z}=1}
  \end{equation*}
  in its boundary and $L$ has minimal length of a Reeb chord equal to the length of a systole of $\Omega$. Indeed, no Reeb chord with boundary on $L$ has distinct endpoints. This counterexample persists if one smooths the corner of $\Omega$ in a reasonable way, say, within the class of convex toric domains containing $\partial D(1)\times D(1)$ in their boundary.
\end{example}

\subsection{Outline of proof}
\label{sec:outline-proof}

From the introduction, it remains to prove Lemma \ref{lemma:aspherical}, Theorem \ref{theorem:main-tech}, and Lemma \ref{lemma:ellipsoid-flex}.

Lemma \ref{lemma:aspherical} will follow from the well-known comparison between string topology and symplectic homology (à la \cite{abouzaid-EMS-2015}), and the fact that the Gutt--Hutchings capacities are defined via $S^{1}$ symplectic homology. 

The first key idea used in the proof of Theorem \ref{theorem:main-tech} is the following:
\begin{bigproposition}[{\cite[Theorem 7]{brocic-cant-arXiv-2026}}]\label{proposition:BC-prop}
  If $L\subset \partial \Omega$ is a Legendrian which does not bound Reeb chords of length up to $a$, then there is an exact embedding of $D_{\epsilon}T^{*}L\times D(a)$ into the interior of $\Omega$, where $D(a)$ is the closed disk of area $a$, and $D_{\epsilon}T^{*}L$ is a small disk subbundle of $T^{*}L$.\hfill$\square$
\end{bigproposition}
The second fact we will need is the monotonicity of the Gutt--Hutchings capacities under exact embeddings:
\begin{bigproposition}[{\cite[Theorem 1.24]{gutt-hutchings-AGT-2018}}]\label{proposition:GH-mono-prop}
  If $\iota:K\to \Omega$ is an exact embedding taking values in the interior of $\Omega$, then $c_{k}(K)\le c_{k}(\Omega)$ holds for all $k$.\hfill$\square$
\end{bigproposition}

The third fact we will need is a special case of the \emph{Cartesian product formula} for the Gutt--Hutchings capacities:
\begin{theorem}\label{theorem:cartesian-product}
  Let $D(a)$ be the closed disk of area $a$ and let $Q$ be a Liouville domain with $c_{i}(\Omega)=\infty$ for $i>i_{0}$. Then there is a Liouville subdomain $K\subset Q\times D(a)$ such that $\liminf c_{k}(K)/k\ge a$.
\end{theorem}
The proofs of Theorem \ref{theorem:cartesian-product} and Lemma \ref{lemma:aspherical} are given in the appendix \S\ref{sec:on-the-gutt-hutchings-capacities}.
\begin{remark}
  The general Cartesian product formula:
  \begin{equation*}
    c_{k}(X_{1}\times X_{2})=\min_{i+j=k}\set{c_{i}(X_{1})+c_{j}(X_{2})}
  \end{equation*}
  is stated for the Ekeland-Hofer capacities in \cite{cieliebak-hofer-latschev-schlenk,gutt-hutchings-AGT-2018}, but it is seemingly not yet proved for the Gutt--Hutchings capacities (at the time of writing). In personal communication with Gutt, the author was informed that a proof via some sort of K\"unneth isomorphism is suspected by the experts.
\end{remark}

The combination of Propositions \ref{proposition:BC-prop}, \ref{proposition:GH-mono-prop} and Theorem \ref{theorem:cartesian-product} will easily yield the proof of Theorem \ref{theorem:main-tech}. The argument is given in \S\ref{sec:proof-main-tech}.

\begin{remark}
  If $\Omega\subset \C^{n}$ is uniformly convex and invariant under the diagonal $S^{1}$-action, then the arguments of Ostrover communicated in \cite[Proposition 1.4]{gutt-hutchings-ramos-JSG-2022} establish that, up to the action of the unitary group, $\Omega$ already lies in $E(c,\infty,\dots)$, where $c$ is the length of a systole. Then, by uniform convexity, $\Omega$ actually lies in $E(c,cN,cN,\dots)$ for $N$ large enough. By Proposition \ref{prop:GH-ellipsoid}:
  \begin{equation*}
    \liminf \frac{c_{k}(\Omega)}{k}\le c\frac{N}{N+(n-1)}<c,
  \end{equation*}
  and so the strong Arnol\'d chord conjecture for E3 Legendrians holds for such uniformly convex and $S^{1}$-invariant domains.
\end{remark}

\subsection{Further questions}
\label{sec:further-questions}

We pose one conjecture and one question.

\begin{conjecture}
  If $\Omega$ is a uniformly convex domain in $\R^{2n}$, $n>1$, then:
  \begin{equation*}
    \liminf_{k\to\infty} \frac{c_{k}(\Omega)}{k}<c_{1}(\Omega).
  \end{equation*}
  We note that this holds in the case $n=2$, by Lemma \ref{lemma:ellipsoid-flex}.
\end{conjecture}

\begin{question}
  Is there an E3 manifold $L$ with $c_{1}(D^{*}TL)<\infty$? More generally, is there any Liouville domain $Q$ such that $c_{1}(Q)<\infty$ but $c_{k}(Q)=\infty$ for $k$ sufficiently large?
\end{question}
\begin{remark}
  In \cite[pp.\,1049]{ganatra-siegel-JDG-2024}, the authors introduce a numerical invariant $\mathbb{F}(Q)$ for a Liouville domain which is equal to the number of Gutt--Hutchings capacities which are finite. Thus the question asks whether there is some $Q$ such that $\mathbb{F}(Q)\ne 0$ and $\mathbb{F}(Q)\ne \infty$.
\end{remark}

\subsection{Acknowledgements}
\label{sec:acknowledgements}

The author thanks Oliver Edtmair for answering questions about his paper \cite{edtmair-gafa-2024}, and Jean Gutt for answering questions about his paper \cite{gutt-hutchings-AGT-2018}. Of course, any mistaken interpretations of their results are solely the responsibility of the present author. The author is also grateful to Filip Broćić for many informative discussions, especially those relating to our joint work \cite{brocic-cant-arXiv-2026} which heavily influenced this note, and to Adi Dickstein and Mark Gudiev for many clarifying discussions about equivariant Floer cohomology, which heavily influenced the exposition in~\S\ref{sec:on-the-gutt-hutchings-capacities}. The author further thanks Pazit Haim-Kislev, Egor Shelukhin, and Shira Tanny for illuminating conversations during the OC60 conference at ETH Zürich in summer 2026, where the majority of this note was written. The author thanks Jun Zhang, Julio Sampietro Christ, Xie Dong for useful comments.

\section{Proofs}
\label{sec:proofs}

The remaining things to prove are:
\begin{itemize}
\item Lemma \ref{lemma:aspherical} and Theorem \ref{theorem:cartesian-product} (deferred to the appendix \S\ref{sec:on-the-gutt-hutchings-capacities});
\item Theorem \ref{theorem:main-tech}, using Propositions \ref{proposition:BC-prop}, \ref{proposition:GH-mono-prop} and Theorem \ref{theorem:cartesian-product};
\item Lemma \ref{lemma:ellipsoid-flex} on the embedding of a uniformly convex domain $\Omega$ into the ellipsoid $E(c,cN)$.
\end{itemize}

\subsection{Proof of Theorem \ref{theorem:main-tech}}
\label{sec:proof-main-tech}

The argument proceeds by contradiction: suppose that $L$ does not bound any Reeb chords of length at most $\liminf c_{k}(\Omega)/k$ (we assume this limit is finite, as per the statement of Theorem \ref{theorem:main-tech}). Then, by a compactness argument, we can find:
\begin{equation}\label{eq:a-strictly-greater}
  a>\liminf_{k\to\infty} c_{k}(\Omega)/k
\end{equation}
such that $L$ does not bound any chords of length at most $a$.

Proposition \ref{proposition:BC-prop} implies there is an embedding $D_{\epsilon}T^{*}L\times D(a)\to \Omega$ taking values in the interior. Proposition \ref{proposition:GH-mono-prop} then implies:
\begin{equation*}
  c_{k}(K)\le c_{k}(\Omega)
\end{equation*}
for any Liouville subdomain $K\subset D_{\epsilon}T^{*}L\times D(a)$. Now we use the assumption that $L$ is E3 to conclude that $c_{i}(D_{\epsilon}T^{*}L)=\infty$ for $i>i_{0}$. Theorem \ref{theorem:cartesian-product} then implies there is a subdomain $K$ such that:
\begin{equation*}
  a\le \liminf c_{k}(K)/k\le \liminf c_{k}(\Omega)/k.
\end{equation*}
This contradicts \eqref{eq:a-strictly-greater}, and completes the proof.\hfill$\square$

\subsection{Proof of Lemma \ref{lemma:ellipsoid-flex}}
\label{sec:proof-flex}

Let $\Omega\subset \R^{4}$ be a uniformly convex domain and denote by $c$ the length of a systole $\gamma$.

In \cite[Corollary 1]{abbondandolo-edtmair-kang-arXiv-2024} it is shown that, for any uniformly convex domain $\Omega$, there is an symplectic embedding of $\Omega$ into a cylinder $E(c,\infty)$. The strategy is as follows, assuming uniform convexity:
\begin{enumerate}[label=(\alph*)]
\item prove that the systole $\gamma$ is a \emph{Hopf orbit}, i.e., $\gamma$ is unknotted and has self-linking number $-1$, using \cite{abbondandolo-edtmair-kang-arXiv-2024} and \cite[Theorem 1.7]{hryniewicz-JSG-2014};
\item use \cite[Proposition 2.8]{florio-hryniewicz-ASENS-2025} to conclude $\gamma$ bounds a $\partial$-strong disk-like global surface of section;
\item\label{step:edtmair} use the construction of \cite{edtmair-gafa-2024}, which assumes $\gamma$ bounds a $\partial$-strong disk-like global surface of section, to conclude a symplectic embedding $\phi:\Omega\to E(c,\infty)$.
\end{enumerate}
The proof of Lemma \ref{lemma:ellipsoid-flex} involves inspection of the last step \ref{step:edtmair}, and shows that the same construction implies $\phi(\Omega)$ lies in an ellipsoid $E(c,cN)$, if $N$ is chosen large enough.

We recall that the construction of the embedding $\phi:\Omega\to E(c,\infty)$ from \cite{edtmair-gafa-2024} is such that $\phi(\Omega)$ is contained in a special domain $A(c,H)\subset E(c,\infty)$, where $H:\R/\Z\times D(c)\to \R$ is a smooth function satisfying:
\begin{enumerate}[label=(O\arabic*)]
\item\label{oliver-1} $H$ is strictly positive in the interior $\mathit{int}(D(c))$,
\item\label{oliver-2} There exists a constant $N>0$ such that, in some neighbourhood of $\partial D(c)$, the function $H$ is given by $H(t,z)=N(c-\pi\abs{z}^{2})$.
\end{enumerate}
The basic properties of $A(c,H)$ are described in \cite[Lemma 3.5]{edtmair-gafa-2024}.

If $H,G$ are two functions satisfying \ref{oliver-1} and \ref{oliver-2} (allowing different constants $N$), it is observed on \cite[pp.\,1419]{edtmair-gafa-2024} that:
\begin{itemize}
\item $H\le G\implies A(c,H)\le A(c,G)$,
\item $A(c,G)=E(c,cN)$ if $G=N(c-\pi\abs{z}^{2})$.
\end{itemize}
Then, for any smooth function $H$ satisfying \ref{oliver-2}, for some constant $N'$, we can obviously pick $N\gg N'$ such that:
\begin{equation*}
  H(t,z)\le N(c-\pi\abs{z}^{2})\text{ everywhere on }\R/\Z\times D(c),
\end{equation*}
and so by the above observations it follows that $A(c,H)\subset E(c,cN)$. This completes the proof of Lemma \ref{lemma:ellipsoid-flex}.\hfill$\square$

\appendix

\section{On the Gutt--Hutchings capacities}
\label{sec:on-the-gutt-hutchings-capacities}

The goals of this section are threefold:
\begin{itemize}
\item review the construction of the Gutt--Hutchings capacities;
\item prove the case of the Cartesian product formula (Theorem \ref{theorem:cartesian-product}) needed in the body of the text;
  \item prove Lemma \ref{lemma:aspherical} that aspherical manifolds are E3.
\end{itemize}

For the history and development of the Gutt--Hutchings capacities, we refer the reader to \cite[\S5]{gutt-hutchings-AGT-2018} and the references therein \cite{viterbo-GAFA-1999,seidel-IP-2008,bourgeois-oancea-IMRN-2017,gutt-JSG-2017}. Other researchers discuss $S^{1}$-equivariant symplectic homology in a manner similar to the way we will, e.g., \cite{zhao-thesis-2016,zhou-jtopol-2021,zhou-adv-math-2022,li-yin-arXiv-2023}. The framework is based on the theory of pseudo-holomorphic curves \cite{gromov-inventiones-1985}.

\subsection{Filtered symplectic homology for a Liouville domain $\Omega$}
\label{sec:filt-s1-equiv}

For the basic set-up of symplectic homology, we follow the notations of \cite{brocic-cant-arXiv-2026}; briefly:
\begin{itemize}
\item $\lambda$ is the Liouville form, $\d\lambda=\omega$;
\item $Z$ is the \emph{Liouville vector field} defined by $\omega(Z,-)=\lambda$;
\item $W$ is the completion of $\Omega$, and $SY\subset W$ is the symplectization end;
\item $r|_{\Omega}=1$ and $\d r(Z)=r$ is the radial function (defined on $SY$);
\item $R=X_{r}$ is the Reeb vector field associated to $\Omega$ (defined on $SY$);
\item $\mathit{Spec}(\Omega)=\set{c:cR\text{ has a 1-periodic orbit}}$;
\item $J$ is an $\omega$-tame almost complex structure which is $Z$ invariant outside of a compact set;
\item given a Hamiltonian function $H$ the associated Hamiltonian vector field $X$ is defined by $\omega(-,X)=\d H$.
\end{itemize}
Contrary to \cite{brocic-cant-arXiv-2026}, in this paper we adopt homological conventions, which means that our Floer cylinders go from left to right; see Figure \ref{fig:left-to-right}. This is to avoid linguistic confusion between $S^{1}$-equivariant homology and cohomology (which involves an additional choice of the direction of the pseudogradient on $BS^{1}$; see \S\ref{sec:borel-equiv-data}).

\begin{figure}[h]
  \centering
  \begin{tikzpicture}[yscale=.7]
    \draw (0,0)--(8,0)arc(-90:90:{0.2 and 0.5})node[pos=0.5,right]{output, $\gamma_{+}$}coordinate(X)--(0,1)arc(90:450:{0.2 and 0.5})node[pos=0.25,left]{input, $\gamma_{-}$};
    \draw[dashed] (X) arc(90:270:{0.2 and 0.5});
    \path (0,0)--node{$\bd_{s}u+J(u)(\bd_{t}u-X_{s,t}(u))=0$}(8,1);
  \end{tikzpicture}
  \caption{Homological conventions.}
  \label{fig:left-to-right}
\end{figure}

\subsection{Borel equivariant data}
\label{sec:borel-equiv-data}

We follow the ideas of \cite{seidel-smith-GAFA-2010} for $G=\Z/2\Z$, which mostly adapt easily to the case of $G=S^{1}$. A discussion of this case is given in \cite[\S2.2]{bourgeois-oancea-IMRN-2017} and \cite[\S2.3]{zhao-thesis-2016}.

\subsubsection{Morse Smale pseudogradient on $BS^{1}$}
\label{sec:pseudogradient-bs1}

\begin{definition}
  Let $V$ be the vector field on $BS^{1}=\C P^{\infty}$ whose time $s$ flow is given by $[z_{0}:z_{1}:z_{2}:\dots]\mapsto [z_{0}:e^{-s}z_{1}:e^{-2s}z_{2}:\dots]$. This is a perfect Morse-Smale pseudogradient.
\end{definition}
We denote by $1:=[1:0:0:\dots]$. Introduce the self-similarity map:
\begin{equation*}
  \tau([z_{0}:z_{1}:\dots])=[0:z_{0}:z_{1}:\dots],
\end{equation*}
and let $u^{-k}=\tau^{k}(1)$. The space of flow lines from $u^{-k}$ to $1$ is a smooth manifold of dimension $2k$. Moreover, $V$ is preserved by the transformation $\tau$ (in particular, flow lines are sent to flow lines).

\subsubsection{Lifted pseudogradient on $ES^{1}$}
\label{sec:pseudogradient-bs1}

In order to lift the pseudogradient $V$ to a pseudogradient $P$ on $S^{\infty}=ES^{1}$, we require a connection on the $S^{1}$-bundle $ES^{1}\to BS^{1}$. To simplify the analysis at asymptotics, we will pick our connection in such a way that it is flat near each critical point $u^{-k}$.

Let $\beta$ be a standard cut-off function, as in \cite[Definition 2.4]{brocic-cant-arXiv-2026}, and let:
\begin{equation*}
  f_{j}([z_{0}:z_{1}:\dots])=\beta(2-2\sum_{i\ne j} \abs{z_{j}^{-1}z_{i}}^{2})
\end{equation*}
The set $\set{f_{j}>0}$ is the unit ball in the affine coordinate chart $z_{j}\ne 0$, and:
\begin{equation}\label{eq:flat-balls}
  B_{j}=\set{\sum \abs{z_{j}^{-1}z_{i}}^{2}\le 1/2}
\end{equation}
is contained in the set where $f_{j}=1$. Introduce:
\begin{equation}\label{eq:the-connection-1-form}
  A=(1-\sum f_{j})\frac{\sum x_{i}\d y_{i}-y_{i}\d x_{i}}{\sum x_{i}^{2}+y_{i}^{2}}+\sum f_{j}\frac{x_{j}\d y_{j}-y_{j}\d x_{j}}{x_{j}^{2}+y_{j}^{2}},
\end{equation}
which is considered as a differential $1$-form on $\C^{\infty}$, with $z_{j}=x_{j}+iy_{j}$. The form $A$ is a \emph{global angular form} for the circle bundle $ES^{1}\to BS^{1}$, and consequently $\xi=\ker A$ defines an Ehresmann connection.

\begin{definition}
  The pseudogradient $P$ on $ES^{1}=S^{\infty}$ is defined to be the unique lift of $V$ satisfying $A(P)=0$. The zeros of $P$ are precisely the circles of the bundle $ES^{1}\to BS^{1}$ living over $u^{-j}$, $j=0,1,\dots$.
\end{definition}

We briefly explain the form of $P$ near the critical circle lying over $u^{-j}$. Above the $z_{j}\ne 0$ affine coordinate chart, $S^{\infty}$ has coordinates:
\begin{equation}\label{eq:coordinates-on-Sinfty}
  \theta_{j}=\frac{z_{j}}{\abs{z_{j}}},\quad w_{0}=\frac{z_{0}}{z_{j}},\quad w_{1}=\frac{z_{1}}{z_{0}},\dots
\end{equation}
where $\theta_{j}$ is $S^{1}$-valued, and $w_{i}$ are $\C$-valued (we omit $w_{j}$ from the list).
\begin{lemma}
  In the coordinates \eqref{eq:coordinates-on-Sinfty}, the vector field $P$ is given by:
  \begin{equation}\label{eq:diag-lin-flow}
    P=\sum_{i=0}^{\infty} (i-j)w_{i}\frac{\partial}{\partial w_{i}}
  \end{equation}
  on the set where $w_{0}^{2}+w_{1}^{2}+\dots\le 1/2$ (we omit $w_{j}$ from this sum).
\end{lemma}
\begin{proof}
  This is a straightforward computation, using \eqref{eq:the-connection-1-form}.
\end{proof}

In particular, near the critical circle lying over $u^{-j}$, the coordinate $\theta_{j}$ is an integral of motion of $P$.

Each flow line $v:\R\to S^{\infty}$ of $P$, satisfying $v'(s)=P(v(s))$, has well-defined asymptotic zeros of $P$.
\begin{definition}
  The space of flow lines of $P$ joining points on the circle over $u^{-j_{-}}$ to points on the circle over $u^{-j_{+}}$ is denoted $\mathscr{P}(j_{-},j_{+})$.
\end{definition}
The $\tau$ invariance of the connection yields canonical isomorphisms:
\begin{equation}\label{eq:tau-identification}
  \tau^{b}:\mathscr{P}(j_{-},j_{+})\simeq \mathscr{P}(j_{-}-j_{+},0),
\end{equation}
i.e., it suffices to only consider the flow lines ending at $1$.

\subsubsection{Borel data}
\label{sec:borel-data}

Let $H$ be an autonomous Hamiltonian on $W$ which commutes with $Z$ and has no $1$-periodic orbits outside of a compact set.
\begin{definition}
  \emph{Borel data} extending $H$ is a family $H_{\eta,t}$, $\eta\in S^{\infty}$, $t\in \R/\Z$, of time-dependent Hamiltonian functions on $W$ such that:
  \begin{itemize}
  \item $H_{\eta,t}=H$ holds outside of a compact set (which can be chosen uniformly on compact subsets of $S^{\infty}$),
  \item $H_{e^{2\pi i\theta}\eta,t}=H_{\eta,t-\theta}$,
  \item $H_{t}=H_{\eta,t}$ is non-degenerate for $\eta=(1,0,\dots)$,
  \item $H_{\tau(\eta),t}=H_{\eta,t}$ where $\tau$ is the self-similarity map.
  \end{itemize}
  We will refer to $H_{t}$, and its generator $X_{t}$, as the \emph{basepoint system}.
\end{definition}

\subsubsection{Moduli spaces associated to Borel data}
\label{sec:moduli-spac-assoc}

Given Borel data $H_{\eta,t}$ as in \S\ref{sec:borel-data}, we define the moduli space $\mathscr{M}(H_{\eta,t})$ of solutions $(v,u)$ to:
\begin{equation*}
  \left\{
    \begin{aligned}
      &v:\R\to S^{\infty},\quad u:\R\times \R/\Z\to W,\\
      &v'(s)=P(v(s)),\\
      &\partial_{s}u+J(u)(\partial_{t}u-X_{v(s),t}(u))=0,\\
      &\textstyle\int \omega(\partial_{s}u,\partial_{t}u-X_{v(s),t}(u))\d s\d t<\infty.
    \end{aligned}
  \right.
\end{equation*}
By standard results (e.g., \cite{floer-comm-math-phys-1989,salamon-notes-1997,bourgeois-oancea-IMRN-2017}) each such solution $(v,u)$ has well-defined asymptotics, in the sense that there are zeros $\eta_{-},\eta_{+}$ of $P$ and $1$-periodic orbits $\gamma_{-}$ and $\gamma_{+}$ of $X_{t}$ such that:
\begin{equation*}  
  \lim_{s\to\infty}u(\pm s,t)=\gamma_{\pm}(t-\theta_{\pm})\quad \lim_{s\to\infty}v(\pm s)=\eta_{\pm},
\end{equation*}
where $\theta_{\pm}$ are the angular coordinates of $\eta_{\pm}$ (suppose $\eta_{\pm}$ lies in the critical circle over $u^{-j_{\pm}}$). Let us extract from such a solution the data of:
\begin{itemize}
\item the asymptotic orbits $\gamma_{-},\gamma_{+}$,
\item the asymptotic indices $j_{-},j_{+}$.
\end{itemize}
These quantities are invariant under the $\R$-action by translations, and the $\R/\Z$-action $(v,u)\mapsto (e^{2\pi i \theta}v,u(s,t-\theta))$. Let us then declare:
\begin{equation}\label{eq:moduli-space}
  \mathscr{M}(H_{\eta,t};\gamma_{-},\gamma_{+};j_{-},j_{+})
\end{equation}
to be the space of solutions with the written asymptotics. As in \eqref{eq:tau-identification} there is an identification:
\begin{equation*}
  \tau^{b}:\mathscr{M}(H_{\eta,t};\gamma_{-},\gamma_{+};j_{-},j_{+})\to\mathscr{M}(H_{\eta,t};\gamma_{-},\gamma_{+};j_{-}-j_{+},0).
\end{equation*}

\begin{definition}\label{definition:regular-borel-data}
  Let us say that Borel data $H_{\eta,t}$ is \emph{regular} provided \eqref{eq:moduli-space} is cut transversally in the usual Floer theoretic sense (as a parametric moduli space of continuation cylinders).
\end{definition}

Definition \ref{definition:regular-borel-data} leaves many things implicit, e.g., how exactly does one locally present $\mathscr{M}(H_{\eta,t};\gamma_{-},\gamma_{+},j_{-},j_{+})$ as the zero set of a smooth map between Banach spaces, etc. These details are standard by now (see the discussion in \cite{zhao-thesis-2016,bourgeois-oancea-IMRN-2017}), and we expect the reader can ``read between the lines,'' with regard to Definition \ref{definition:regular-borel-data}.

\subsection{Orientation lines}
\label{sec:orientation-lines}

The Gutt--Hutchings capacities are defined using $S^{1}$-equivariant symplectic homology with rational coefficients. For this reason, some framework for extracting $\pm$ signs is required to define the Floer complexes. We follow \cite[pp.\,290]{abouzaid-EMS-2015} and \cite[\S2.1.2]{zhao-thesis-2016} and use the language of \emph{orientation lines}; see also \cite[\S2.4]{cant-kilgore-zhang-arXiv-2026}, \cite[\S C.13]{pardon-GT-2016}. Another approach is via the theory of coherent orientations, as in \cite{bourgeois-oancea-IMRN-2017}. The necessary analytical results are based on the kernel gluing results of \cite{floer-hofer-math-z-1993}.

Each orbit $\gamma$ of $X_{t}$ is associated a free rank-one $\Z$-module $\mathfrak{o}(\gamma)$. For any continuation cylinder $u$ with non-degenerate asymptotics:
\begin{itemize}
\item $\gamma_{-}(t-\theta_{-}),\gamma_{+}(t-\theta_{+})$,
\end{itemize}
e.g., as appear in \eqref{eq:moduli-space}, there is a canonical isomorphism:
\begin{equation}\label{eq:isomorphism-orientation-line}
  \mathfrak{o}(D_{u})\simeq \mathfrak{o}(\gamma_{-})\otimes \mathfrak{o}(\gamma_{+}),
\end{equation}
where $\mathfrak{o}(D_{u})$ is the orientation line of the Fredholm determinant of the linearized operator of $u$. These canonical isomorphisms are compatible with the usual gluing maps of Floer theory. The way this is used is as follows:
\begin{itemize}
\item the count of the rigid orbits of \eqref{eq:moduli-space} is valued in $\mathfrak{o}(\gamma_{-})\otimes \mathfrak{o}(\gamma_{+})$.
\item the count of ends of the $1$-dimensional component in the space of orbits of \eqref{eq:moduli-space} is also valued in $\mathfrak{o}(\gamma_{-})\otimes \mathfrak{o}(\gamma_{+})$, and this count vanishes.
\end{itemize}
Because both counts are defined using the isomorphisms \eqref{eq:isomorphism-orientation-line}, one is able to perform the usual Floer theoretic argument and prove various equalities between the rigid counts using the 1-dimensional moduli spaces.

For the above scheme to work with the moduli space \eqref{eq:moduli-space}, we should pick canonical orientations for the spaces of flow lines $\mathscr{P}$, which can be done using the complex orientations on the stable manifolds of $V$.

\subsection{The $S^{1}$-equivariant complex}
\label{sec:s1-equiv-compl}

An important ingredient in the definition of the Gutt--Hutchings capacities is the $\Q[[u]]$-module structure on the filtered $S^{1}$-equivariant homology. 

The chain complex underlying the $S^{1}$-equivariant homology is not a free $\Q[[u]]$ module, but is rather a direct sum of finitely many copies of:
\begin{equation}
  \Q[u^{-1}]:=\Q[u^{-1},u]]/u\Q[[u]].
\end{equation}
This should be considered as a $\Q[[u]]$-module, (not as a $\Q[u^{-1}]$-module). See \cite[\S3.1]{ganatra-siegel-JDG-2024} for similar notation.
\begin{definition}
  For regular Borel data $H_{\eta,t}$, define:
  \begin{equation*}
    \textstyle \mathit{CF}^{\mathit{eq}}(H_{\eta,t})=\bigoplus \Q[u^{-1}]\otimes \mathfrak{o}(\gamma),
  \end{equation*}
  where the direct sum is over $1$-periodic orbits of the basepoint system $H_{t}$.
\end{definition}

The differential on the equivariant complex is given by a formula of the form:
\begin{equation*}
  d_{\mathit{eq}}=d_{0}+ud_{1}+u^{2}d_{2}+\dots,
\end{equation*}
where:
\begin{equation*}
  d_{j}:\textstyle\bigoplus \mathfrak{o}(\gamma_{-})\to \bigoplus \mathfrak{o}(\gamma_{+})
\end{equation*}
is an endomorphism extended trivially to the tensor product with $\Q[u^{-1}]$. The term $d_{j}|_{\mathfrak{o}(\gamma_{-})}$ is defined by counting the rigid $\R\times \R/\Z$-orbits of: $$\mathscr{M}(H_{\eta,t};\gamma_{-},\gamma_{+};j,0)$$ as an element of $\mathfrak{o}(\gamma_{-})\otimes \mathfrak{o}(\gamma_{+})\simeq \mathit{Hom}(\mathfrak{o}(\gamma_{-}),\mathfrak{o}(\gamma_{+}))$, as described in \S\ref{sec:orientation-lines}.
\begin{claim}
  The differential squares to zero, $d_{\mathit{eq}}^{2}=0$.
\end{claim}
\begin{proof}
  See \cite[\S2.2]{bourgeois-oancea-IMRN-2017} and \cite[\S2.3]{zhao-thesis-2016}.
\end{proof}

The homology of the complex $(\mathit{CF}^{\mathit{eq}}(H_{\eta,t}),d_{\mathit{eq}})$ is denoted $\mathit{HF}^{\mathit{eq}}(H_{\eta,t})$.

\subsubsection{Continuation maps}
\label{sec:continuation-maps}

If $H\ge K$ and $H_{\eta,t}$ is regular Borel data extending $H$, while $K_{\eta,t}$ is regular Borel data extending $K$, then there is a continuation chain map:
\begin{equation*}
  \mathit{CF}^{\mathit{eq}}(H_{\eta,t})\to \mathit{CF}^{\mathit{eq}}(K_{\eta,t})
\end{equation*}
defined by counting solutions $(v,u)$ of the equation:
\begin{equation*}
  \left\{
    \begin{aligned}
      &v:\R\to S^{\infty},\quad u:\R\times \R/\Z\to W,\\
      &v'(s)=P(v(s)),\quad \partial_{s}u+J(u)(\partial_{t}u-X_{v(s),s,t}(u))=0,\\
      &\textstyle\int \omega(\partial_{s}u,\partial_{t}u-X_{v(s),s,t}(u))\d s\d t<\infty,
    \end{aligned}
  \right.
\end{equation*}
where $X_{\eta,s,t}$ is the generator of (a small perturbation of) the monotone homotopy $(1-\beta(s))H_{\eta,t}+\beta(s)K_{\eta,t}$, where $\beta$ is as in \S\ref{sec:pseudogradient-bs1}.

The resulting chain map is independent, up to chain homotopy, of the choice of perturbation. Moreover, the composition of continuation maps is again a continuation map, and the continuation map associated to the constant homotopy, when $H_{\eta,t}=K_{\eta,t}$, is equal to the identity map. This is described in \cite[\S2.3]{zhao-thesis-2016}, and the ideas follow Floer's arguments in \cite{floer-comm-math-phys-1989}.

Consequently, one has a well-defined invariant $\mathit{HF}^{\mathit{eq}}(H)$ which depends only on the ideal restriction $H$.
\begin{definition}
  The filtered $S^{1}$-equivariant symplectic cohomology group is:
  \begin{equation*}
    \mathit{SH}_{c}^{\mathit{eq}}(\Omega):=\mathit{HF}^{\mathit{eq}}(-cr),
  \end{equation*}
  and is only defined for $c\not\in\mathit{Spec}(\Omega)$.
\end{definition}
The above discussion produces continuation maps:
\begin{itemize}
\item $\mathit{SH}^{\mathit{eq}}_{c}(\Omega)\to \mathit{SH}^{\mathit{eq}}_{c'}(\Omega)$ if $c<c'$,
\end{itemize}
giving $c\mapsto \mathit{SH}_{c}^{\mathit{eq}}(\Omega)$ the structure of a persistence module.

\subsection{The PSS morphism}
\label{sec:pss-morphism}

Part of the theory supporting the Gutt--Hutchings capacities is the existence of an exact triangle of $\Q[[u]]$-modules:
\begin{equation}\label{eq:exact-triangle}
  \begin{tikzpicture}[baseline={(0,-0.3)}]
    \path (30:1)node[right](A){$\mathit{SH}_{c}^{\mathit{eq}}(\Omega)$}--(150:1)node[left](B){$\mathit{SH}^{\mathit{eq},+}_{c}(\Omega)$}--(270:1)node(C){$\mathit{HM}(W,\Q[u^{-1}])$};
  \draw[->] (A)--node[above]{}(B);
  \draw[->] (B)--node[pos=0.3,below left]{$\delta$}(C);
  \draw[->] (C)--node[pos=0.7,below right]{$\mathit{PSS}$}(A);
  \end{tikzpicture}
\end{equation}
where $c>0$ and $\mathit{HM}(W,\Q[u^{-1}])$ is the Morse homology of a suitable class of pseudogradients on $W$ with coefficients in $\Q[u^{-1}]$. Let us comment that:
\begin{itemize}
\item $\mathit{HM}(W,\Q[u^{-1}])=\mathit{HM}(W,\Q)\otimes \Q[u^{-1}]$,
\end{itemize}
and that $\mathit{HM}(W,\Q)$ contains a distinguished class denoted $1$ (the definition will be recalled in \S\ref{sec:morse-homology-conventions}). We avoid discussion of $\mathit{SH}^{G,+}(\Omega)$ by focusing on the other terms in the sequence.

For any $m\in \mathit{HM}(W,\Q)$, we can also speak of the classes:
\begin{equation*}
  u^{d}m=m\otimes u^{d}\in \mathit{HM}(W,\Q[u^{-1}])\text{ for }d\in \Z,
\end{equation*}
so that $u(u^{d}m)=u^{d+1}m$ and $u^{d}m=0$ for $d>0$.

For the PSS map\footnote{The name PSS comes from the paper \cite{piunikhin-salamon-schwarz-1996}.} appearing in \eqref{eq:exact-triangle} in a similar context, see \cite[\S4.1]{zhao-thesis-2016}; we will explain how this map is defined in \S\ref{sec:pss-map-s1} below.

With regard to diagram \eqref{eq:exact-triangle}, the $k$th \emph{Gutt--Hutchings capacity} is:
\begin{equation}\label{eq:standard-GH}
  c_{k}(\Omega):=\inf\set{c>0:\delta(u^{k-1}\zeta)=1\text{ for some }\zeta\in \mathit{SH}_{c}^{\mathit{eq},+}(\Omega)}.
\end{equation}
We will reformulate this definition to only refer to $\mathit{SH}^{\mathit{eq}}_{c}(\Omega)$ and $\mathit{PSS}$.

From basic algebra, the following are seen to be equivalent:
\begin{itemize}
\item $\delta(u^{k-1}\zeta)=1$,
\item $u^{k-1}(\delta(\zeta)-u^{1-k}1)=0$,
\item $\delta(\zeta)=u^{1-k}1+u^{2-k}m_{1}+u^{3-k}m_{2}+\dots$,
\item $\mathit{PSS}(u^{1-k}1+u^{2-k}m_{1}+\dots)=0$.
\end{itemize}
\begin{definition}\label{definition:GH}
  The $k$th \emph{Gutt--Hutchings capacity} is defined by:
  \begin{equation*}
    c_{k}(\Omega)=\inf\set{c:\mathit{PSS}(u^{1-k}1+u^{2-k}m_{1}+\dots)=0\text{ in }\mathit{SH}_{c}^{\mathit{eq}}(\Omega)},
  \end{equation*}
  where the vanishing should hold for some choice of $m_{i}$. By the above discussion, this is equivalent to the standard definition \eqref{eq:standard-GH}.
\end{definition}
The goals of this section are to describe the construction of the map $\mathit{PSS}$ and to recall its basic properties.

\subsubsection{Morse homology conventions}
\label{sec:morse-homology-conventions}

Let $G$ be a Morse--Smale pseudogradient on $W$ which agrees with $Z$ outside of a compact set. Define:
\begin{equation*}
  \mathit{CM}(G):=\textstyle\bigoplus \mathfrak{o}(x)
\end{equation*}
where the direct sum is over zeroes $x$ of $G$, and $\mathfrak{o}(x)$ is the orientation line for the stable manifold at $x$. The Morse differential:
\begin{equation*}
  d:\mathit{CM}(G)\to \mathit{CM}(G)
\end{equation*}
is computed by counting flow lines of $G$. Each zero $x$ with trivial stable manifold has a canonical generator $1_{x}\in \mathfrak{o}(x)$, since the zero vector space is canonically oriented.
\begin{definition}
  The \emph{unit cycle} $1\in \mathit{CM}(G)$ is defined to be the direct sum of $1_{x}$ over all zeros $x$ with trivial stable manifold. It is a fundamental result in Morse theory that $d(1)=0$.
\end{definition}

The above defines $\mathit{CM}(G)$ as the Morse complex ``over the integers.'' To obtain the Morse complex $\mathit{CM}(G)$ over another ring $S$, we simply tensor $\mathit{CM}(G)$ with the ring $S$, and extend the differential so that it is $S$-linear. The homology of $\mathit{CM}(G)\otimes S$ is denoted $\mathit{HM}(G,S)$.

Finally, we recall that for any two such pseudogradients, there are canonical chain homotopy equivalences relating their Morse complexes (called continuation isomorphisms); the resulting invariant of $W$ is denoted $\mathit{HM}(W,S)$.

\subsubsection{The PSS map in $S^{1}$-equivariant symplectic homology}
\label{sec:pss-map-s1}

Fix:
\begin{itemize}
\item $G$ as in \S\ref{sec:morse-homology-conventions}, 
\item $c>0$ so $c\not\in \mathit{Spec}(\Omega)$, and,
\item regular Borel data $H_{\eta,t}$ extending $H=-cr$.
\end{itemize}
The goal in this section is to define chain maps:
\begin{equation*}
  \mathit{PSS}:\mathit{CM}(G)\otimes \Q[u^{-1}]\to \mathit{CF}^{eq}(H_{\eta,t})
\end{equation*}
representing the map appearing in \eqref{eq:exact-triangle}.

Similarly to \S\ref{sec:continuation-maps}, introduce the moduli space $\mathscr{N}(H_{\eta,t})$ of solutions:
\begin{equation*}
  \left\{
    \begin{aligned}
      &v:\R\to S^{\infty},\quad u:\R\times \R/\Z\to W,\\
      &v'(s)=P(v(s)),\quad \partial_{s}u+J(u)(\partial_{t}u-X_{v(s),s,t}(u))=0,\\
      &\textstyle\int \omega(\partial_{s}u,\partial_{t}u-X_{v(s),s,t}(u))\d s\d t<\infty,
    \end{aligned}
  \right.
\end{equation*}
where now $X_{\eta,s,t}$ is the generator of (a small perturbation of) $\beta(s)H_{\eta,t}$. As in \S\ref{sec:moduli-spac-assoc}, we can extract from a solution $(v,u)$ asymptotics $\gamma_{+},j_{-},j_{+}$, but now the left asymptotic $u(-\infty)$ is a removable singularity in $W$, as in \cite{piunikhin-salamon-schwarz-1996}.

Let us denote by $\mathscr{N}(H_{\eta,t};\gamma_{+};j_{-},j_{+})$ the moduli space of solutions with the written asymptotics, and consider the evaluation at the left asymptotic as a smooth map valued in $W$. By picking the perturbation term in the definition of $X_{\eta,s,t}$ sufficiently generically, we can ensure that this evaluation map is transverse to the unstable manifolds $W_{x}$ of zeros $x$ of $G$. Thus we can define:
\begin{equation*}
  \mathfrak{P}_{j}:\mathit{CM}(G)\to \mathit{CF}(H_{t})
\end{equation*}
by the formula:
\begin{equation*}
  \mathfrak{P}_{j}|_{\mathfrak{o}(x)\to \mathfrak{o}(\gamma)}=\#\set{(v,u)\in \mathscr{N}(H_{\eta,t};\gamma;j,0):u(-\infty)\in W_{x}}/(\R/\Z).
\end{equation*}
In words, we count the rigid solutions of the PSS equation (modulo the circle action), which lie above a flow line $v$ joining the $u^{-j}$ circle to the $1$ circle.

Define:
\begin{equation*}
  \mathfrak{P}_{\mathit{eq}}:=\sum u^{j}\mathfrak{P}_{j}.
\end{equation*}
where we extend each $\mathfrak{P}_{j}$ to a map which is linear over $\Q[u^{-1}]$.
\begin{claim}
  The map $\mathfrak{P}_{\mathit{eq}}$ is a chain map, i.e., $\mathfrak{P}_{\mathit{eq}}d=d_{\mathit{eq}}\mathfrak{P}_{\mathit{eq}}$. Moreover, the chain homotopy class of $\mathfrak{P}_{\mathit{eq}}$ is independent of the choice of perturbation, and it commutes with the chain homotopy class of continuation maps.
\end{claim}
\begin{proof}
  The follows the same logic as why $d_{eq}^{2}=0$ \cite[\S2.2]{bourgeois-oancea-IMRN-2017}, and why continuation maps are chain maps for $d_{eq}$ \cite[pp.\,25]{zhao-thesis-2016}, together with the usual PSS arguments \cite{piunikhin-salamon-schwarz-1996,frauenfelder-schlenk-israeljm-2007}, and, of course, the arguments of \cite{floer-comm-math-phys-1989}. In fact, it is not hard to see that, if one picks $H_{\eta,t}$ appropriately then our map agrees with the one considered by \cite[Proposition 4.1.2]{zhao-thesis-2016}.
\end{proof}
The map induced on homology by $\mathfrak{P}_{eq}$ is denoted $\mathit{PSS}$ in \eqref{eq:exact-triangle}. This concludes our review of the Gutt-Hutchings capacity, and the setting of the stage for the proof of Theorem \ref{theorem:cartesian-product}.

\subsection{The K\"unneth map and the proof of Theorem \ref{theorem:cartesian-product}}
\label{sec:kunneth-map}

The key Floer theoretic operation required in this section is a ``K\"unneth morphism.'' We will use the same framework as \cite[\S 2.2]{brocic-cant-arXiv-2026}. We briefly recall the notation and set-up there:
\begin{itemize}
\item the Liouville domain is $Q\times D(a)$, with completion $W\times \C$;
\item the radial coordinate on $Q$ is $r_{1}$;
\item the radial coordinate on $\C$ is $r_{2}=a^{-1}\pi\abs{z}^{2}$;
\item $J_{1}$ is an almost complex structure on $W$, invariant under the Liouville flow outside of $Q$;
\item $J_{2}$ is the standard almost complex structure on $\C$;
\item $f(r)$ is a convex, non-decreasing function which is constant on $[0,\frac{1}{2}]$ and satisfies $f(r)=r$ for $r\ge 1$;
\item $\ell(r)=\log(f(r))$ is a smoothing of $\max\set{\log(r),0}$.
\end{itemize}
In \cite[\S2.2.1]{brocic-cant-arXiv-2026} we had introduced the special almost complex structure:
\begin{equation*}
  J:=Z^{*}_{-\ell(r_{2})}J_{1}\oplus J_{2}
\end{equation*}
on $W\times \C$, which is Liouville equivariant outside $\set{r_{1}\le 1}\cap \set{r_{2}\le 1}$ but is split on the strip $\set{r_{2}\le 1/2}$.

In \cite[\S2.2.2]{brocic-cant-arXiv-2026} we had introduced a special class of Hamiltonian systems. We modify this class slightly for use with the $S^{1}$-equivariant theory.

\begin{definition}\label{definition:admissible-for-K}
  Fix a small number $\delta>0$. Borel data $H_{\eta,t}$ on $W\times \C$ is \emph{admissible for the Künneth map} for $Q\times D(a)$ with:
  \begin{itemize}
  \item slopes $0<b_{1}$ and $0<b_{2}$ satisfying $b_{1}\not\in \mathit{Spec}(Q)$ and $b_{2}\not\in a\Z$, and
  \item perturbation $P_{\eta,t}$ supported in $Q$,
  \end{itemize}
  if it is given by:
  \begin{equation*}
    H_{\eta,t}=-b_{1}e^{\ell(r_{2})}\delta f(\delta^{-1}e^{-\ell(r_{2})}r_{1})-b_{2}r_{2}+\beta(2-2r_{2})P_{\eta,t}
  \end{equation*}
  and:
  \begin{enumerate}[label=(K\arabic*)]
  \item\label{item:K1} all orbits of the basepoint system $H_{t}$ lie in $W\times 0$, and,
  \item\label{item:K2} $H_{\eta,t}$ is regular Borel data.
  \end{enumerate}  
\end{definition}
\begin{lemma}
  For any $\delta>0$ small enough, and any slopes $b_{1},b_{2}$ as in Definition \ref{definition:admissible-for-K}, there exist perturbation terms $P_{\eta,t}$ such that the resulting $H_{\eta,t}$ satisfies \ref{item:K1} and \ref{item:K2}.
\end{lemma}
\begin{proof}
  The argument is the same as the one given in \cite[Lemma 2.8]{brocic-cant-arXiv-2026}, with one small modification: in the notation given there, one can make the analogue of ``$bS$'' arbitrarily close to $b_{2}$, which is not a period of $X_{r_{2}}$, and so no orbits lying above the region $r_{2}\ne 0$ lie on one-periodic orbits. 
\end{proof}

\begin{lemma}\label{lemma:kunneth-lemma-technical}
  For $\delta>0$ and $P_{\eta,t}$ small enough, and $b_{1},b_{2}$ as in Definition \ref{definition:admissible-for-K}, there is an isomorphism of complexes:
  \begin{equation}\label{eq:kunneth-iso}
   \mathfrak{K}:\mathit{CF}^{eq}(H_{\eta,t})\to \mathit{CF}^{eq}(-b_{1}\rho_{1}+P_{\eta,t})\otimes_{\Q[[u]]}x_{k}\Q[u^{-1}]
  \end{equation}
  where $x_{k}$ is a formal generator corresponding to the orbit of $-b_{2}X_{r_{2}}$ located at the origin, with the label $k=\lfloor b_{2}/a\rfloor$, and $\rho_{1}=e^{\ell(0)}\delta f(\delta^{-1}e^{-\ell(0)}r_{1})$.

  Secondly, for $b_{2}/a\in (0,1)$, there is a commutative diagram:
  \begin{equation*}
    \begin{tikzcd}
     \mathit{CM}(W\times\C,\Q[u^{-1}])\arrow[d,"\mathfrak{P}"]\arrow[r,"\mathfrak{K}"]&\mathit{CM}(W,\Q[u^{-1}])\arrow[d,"\mathfrak{P}\otimes \id"]\\
\mathit{CF}^{eq}(H_{\eta,t})\arrow[r,"\mathfrak{K}"]&\mathit{CF}^{eq}(-b_{1}\rho_{1}+P_{\eta,t})\otimes_{\Q[[u]]}x_{0}\Q[u^{-1}],
    \end{tikzcd}
  \end{equation*}
  where $\mathit{PSS}$ is the map represented by the chain maps $\mathfrak{P}_{\mathit{eq}}$ defined in \S\ref{sec:pss-map-s1}.
  
  Finally, if $H'_{\eta,t}$ is defined similarly for $b_{2}'>b_{2}$, and the same $b_{1}$ and $P_{\eta,t}$, then the continuation map:
  \begin{equation*}
    \mathfrak{c}:\mathit{CF}^{\mathit{eq}}(H_{\eta,t})\to \mathit{CF}^{\mathit{eq}}(H_{\eta,t}')
  \end{equation*}
  satisfies:
  \begin{equation}\label{eq:stated-fact}
    \text{if $\d_{eq}\zeta=0$ and $\mathfrak{c}(\zeta)=\d_{eq} \xi$, then $u^{\tau}\zeta=\d_{eq} \mu$ for some cycle $\mu$,}
  \end{equation}
  where $\tau=k'-k$.
\end{lemma}
\begin{proof}
  The idea for the definition of $\mathfrak{K}$ is the same as the one given in \cite[\S2.2]{brocic-cant-arXiv-2026}, and one which is common to most treatments of the K\"unneth isomorphism in Morse type theories.

  One uses \ref{item:K1} to show the orbits of $H_{t}$ are pairs $(\gamma,x_{k})$ where:
  \begin{itemize}
  \item $\gamma$ is an orbit of $-b_{1}\rho_{1}+P_{t}$, and,
  \item $x_{k}$ is the constant orbit of $-b_{2}r_{2}$ located at the origin.
  \end{itemize}
  This directly identifies the $\Q[[u]]$-modules underlying the complexes (one also shows that the orientation lines split $\mathfrak{o}(\gamma,x_{k})=\mathfrak{o}(\gamma)\otimes \mathfrak{o}(x_{k})$).

  The next step is to compute the differential with respect to this identification. By the compactness argument given in \cite[Lemma 2.9]{brocic-cant-arXiv-2026}, one proves that for $\delta$ and $P_{\eta,t}$ small enough, all solutions contributing to the differential remain in the region $r_{2}\le 1/2$ and therefore satisfy the split equation:
  \begin{equation*}
    \left\{
      \begin{aligned}
        &v:\R\to S^{\infty}\text{ flow line for $P$},\quad u=(u_{1},u_{2}):\R\times \R/\Z\to W\times \C,\\
        &\partial_{s}u_{1}+J_{1}(u_{1})(\partial_{t}u_{1}-X_{v(s),t}(u_{1}))=0,\\
        &\partial_{s}u_{2}+J_{2}(u_{2})(\partial_{t}u_{2}-b_{2}X_{r_{2}}(u_{2}))=0,\\
      \end{aligned}
    \right.
  \end{equation*}
  (with the finite energy condition), where we relabel:
  \begin{itemize}
  \item $X_{\eta,t}$ as the generator of $-b_{1}\rho_{1}+P_{\eta,t}$,
  \item $J_{1}$ as $J|_{W\times \set{0}}$ (this is a slight abuse of notation, since $J_{1}$ appears in the definition of $J$).
  \end{itemize}
  This implies that the identification $\mathfrak{K}$ satisfies:
  \begin{equation*}
    \mathfrak{K}\circ d_{\mathit{eq}}=(d_{\mathit{eq}}\otimes 1)\circ \mathfrak{K},
  \end{equation*}
  i.e., it is an isomorphism of chain complexes.

  The second part concerning the PSS maps follows a similar reasoning (a compactness argument proves solutions remain concentrated at $W\times \set{0}$, and the equation splits as above). For this step, one should pick the pseudogradient on $W\times \C$ to be a stabilization of a pseudogradient on $W$. There is an additional step concerning the linearization of the PSS equation, where one uses $b_{2}/a\in (0,1)$; this step is described in \cite[Lemma 2.11]{brocic-cant-arXiv-2026}.

  The final part a bit more subtle, especially if $b_{2}$ and $b_{2}'$ lie in different components of $\R\setminus a\Z$, as we need to invoke the additional perturbations in the definition of the continuation maps, and the solutions can no longer be guaranteed to remain in $W\times \set{0}$.

  By perturbing slightly the equation on the second factor, one arranges that all solutions solve the equation:
  \begin{equation}\label{eq:continuation-map-eq}
    \left\{
      \begin{aligned}
        &v:\R\to S^{\infty}\text{ flow line of $P$},\quad u=(u_{1},u_{2}):\R\times \R/\Z\to W\times \C,\\
        &\partial_{s}u_{1}+J_{1}(u_{1})(\partial_{t}u_{1}-X_{v(s),t}(u_{1}))=0,\\
        &\partial_{s}u_{2}+J_{2}(u_{2})(\partial_{t}u_{2}-b(s)X_{r_{2}}(u_{2})+V_{v(s),s,t}(u_{2}))=0,
      \end{aligned}
    \right.
  \end{equation}
  where $b(s)$ interpolates from $b_{2}$ to $b_{2}'$ and where $V_{\eta,s,t}$ is an $s$-dependent perturbation satisfying $V_{e^{2\pi i \theta}\eta,s,t}=V_{\eta,s,t-\theta}$. Here $X_{\eta,t}$ and $J_{1}$ are as in the first part of the proof.
  
  To prove the statement involving $u^{\tau}$, we recall in a bit more detail how the continuation map is actually defined:
  \begin{equation*}
    \mathfrak{c}=\mathfrak{c}_{0}+u\mathfrak{c}_{1}+u^{2}\mathfrak{c}_{2}+\dots,
  \end{equation*}
  where $\mathfrak{c}_{j}$ counts those solutions of \eqref{eq:continuation-map-eq} for which $v$ joins the critical circle over $u^{-j}$ to the critical circle over $1$. If $\tau=k'-k$ as in the statement, then we claim:
  \begin{enumerate}[label=(\roman*)]
  \item\label{roman-fact-1} $\mathfrak{c}_{0}=\dots=\mathfrak{c}_{d-1}=0$,
  \item\label{roman-fact-2} $\mathfrak{c}_{\tau}$ is a quasi-isomorphism for the non-equivariant differential.
  \end{enumerate}
  It follows easily from \ref{roman-fact-1} that:
  \begin{equation*}
    \tilde{\mathfrak{c}}=\mathfrak{c}_{\tau}+u\mathfrak{c}_{\tau+1}+u^{2}\mathfrak{c}_{\tau+2}+\dots
  \end{equation*}
  is a chain map for $d_{\mathit{eq}}$ and that $\mathfrak{c}=u^{\tau}\tilde{\mathfrak{c}}$.
  
  By a spectral sequence argument, $\tilde{\mathfrak{c}}$ is a quasi-isomorphism for $d_{\mathit{eq}}$, because of \ref{roman-fact-2}. Thus the stated fact \eqref{eq:stated-fact} follows.

  It remains only to prove \ref{roman-fact-1} and \ref{roman-fact-2}. The key idea is to split the equation \eqref{eq:continuation-map-eq} into two parts; the first part is the space $\mathscr{D}$ of pairs $(v,u_{2})$ solving:
  \begin{equation*}
    \left\{
      \begin{aligned}
        &v:\R\to S^{\infty}\text{ flow line of $P$},\quad u_{2}:\R\times \R/\Z\to \C,\\
        &\partial_{s}u_{2}+J_{2}(u_{2})(\partial_{t}u_{2}-b(s)X_{r_{2}}(u_{2})+V_{v(s),s,t}(u_{2}))=0,
      \end{aligned}
    \right.
  \end{equation*}
  This moduli space is relatively well-understood, since there is only one possible orbit $x_{k}$ at the input and only one possible orbit $x_{k'}$ at the output, and both orbits have a well-understood Conley-Zender index so the Fredholm index of $u_{2}$ is known to be $2(k-k')$. This dimensional argument implies that there are no solutions (assuming a generic perturbation) if $v$ joins $u^{-j}$ to $1$ for $j<\tau$; this proves \ref{roman-fact-1}.

  Similarly, the only rigid solutions of this equation are those when $v$ joins $u^{-\tau}$ to $1$. Since we know the Gutt-Hutchings capacities of the disk are unbounded, the count of these rigid solutions must be a non-zero rational number $q$ (otherwise the $S^{1}$-equivariant continuation map would be identically zero, and the Gutt-Hutchings capacities would be bounded). Having established, this, we observe that \eqref{eq:continuation-map-eq} counts pairs $((v,u_{2}),u_{1})$ such that $(v,u_{2})\in \mathscr{D}$ and $u_{1}$ solves:
  \begin{equation*}
    \partial_{s}u_{1}+J_{1}(u_{1})(\partial_{t}u_{1}-X_{v(s),t}(u_{1}))=0,
  \end{equation*}
  which, for any fixed $(v,u_{2})\in \mathscr{D}$ is just the continuation map equation. Thus we conclude that $q^{-1}\mathfrak{c}_{\tau}$ is in the chain homotopy class of a continuation map, which is a quasi-isomorphism. This proves \ref{roman-fact-2}.
\end{proof}

\begin{proof}[Proof of Theorem \ref{theorem:cartesian-product}]
  The first step is to use the second part of Lemma \ref{lemma:kunneth-lemma-technical} to show that the following equation cannot be solved in $\mathit{CF}^{eq}(H_{\eta,t})$:
  \begin{equation*}
    \mathfrak{P}(u^{-i_{0}}1+u^{1-i_{0}}m_{1}+\dots)=d_{eq}(\mu)
  \end{equation*}
  if $H_{\eta,t}$ is as in Definition \ref{definition:admissible-for-K} for any slope $b_{1}$ and with slope $b_{2}/a\in (0,1)$.

  Now we use the third part of Lemma \ref{lemma:kunneth-lemma-technical} to conclude that:
  \begin{equation*}
    \mathfrak{P}(u^{-i_{0}-k}1+u^{1-i_{0}-k}m_{1}+\dots)=d_{eq}(\mu)
  \end{equation*}
  cannot be solved, if $H'_{\eta,t}$ is as in Definition \ref{definition:admissible-for-K} with slope $b_{2}'/a\in (k,k+1)$. Indeed, if we could solve this equation, then we conclude that:
  \begin{equation*}
    \mathfrak{P}(u^{-i_{0}-k}1+u^{1-i_{0}-k}m_{1}+\dots)\in \mathit{CF}^{\mathit{eq}}(H_{\eta,t})
  \end{equation*}
  is mapped to $d_{eq}(\mu)$ by the continuation map $\mathfrak{c}$, and so by the third part of Lemma \ref{lemma:kunneth-lemma-technical} we conclude that:
  \begin{equation*}
    u^{k}\mathfrak{P}(u^{-i_{0}-k}1+u^{1-i_{0}-k}m_{1}+\dots)=\d_{eq}(\xi)\text{ holds in }\mathit{CF}^{\mathit{eq}}(H_{\eta,t}),
  \end{equation*}
  but this contradicts the first paragraph of the proof.

  Let us now take $b=b_{1}=b_{2}'>ka$. Then $H_{\eta,t}'$ agrees with $-br$ where
  \begin{equation*}
    r=e^{\ell(r_{2})}\delta f(\delta^{-1}e^{-\ell(r_{2})}r_{1})+r_{2},
  \end{equation*}
  producing $K=\set{r\le 1}\subset Q\times D(a)$, as in \cite[Definition 2.6]{brocic-cant-arXiv-2026}. We have therefore shown that:
  \begin{equation*}
    c_{k+i_{0}+1}(K)>ka\implies \liminf_{k\to\infty}\frac{c_{k}(K)}{k}\ge a,
  \end{equation*}
  as desired.
\end{proof}

\subsection{Proof of Lemma \ref{lemma:aspherical}}
\label{sec:proof-lemma-refl}

Finally, we briefly review the proof of Lemma \ref{lemma:aspherical}. The key result we will is need is the main result of \cite[Chapter 12]{abouzaid-EMS-2015} is that there is a local system $\mathfrak{L}$ on the free loop space, which restricts to the orientation local system $\mathfrak{o}(TL)$ when $\mathfrak{L}$ is pulled back to the space of constant loops, and a commutative diagram:
\begin{equation}\label{eq:abouzaid-diagram}
  \begin{tikzcd}
    H_{*}(L,\mathfrak{o}(TL))\arrow[d,"\simeq"]\arrow[r,"\iota"]&H_{*}(\Lambda L,\mathfrak{L})\arrow[d,"\simeq"]\\
      \mathit{HM}(T^{*}L,\Q)\arrow[r,"\mathit{PSS}"]&\mathit{SH}(T^{*}L).
  \end{tikzcd}
\end{equation}
Moreover, the commutative diagram respects the decomposition into summands corresponding to free homotopy classes.

\begin{lemma}
  If $L$ is aspherical, then the map $\iota$ in \eqref{eq:abouzaid-diagram} is an isomorphism onto the summand of contractible loops.
\end{lemma}
\begin{proof}
  This follows from \cite[Lemma 5.2]{latschev-EMS-2015}.
\end{proof}

The argument in \cite[Claim 1.7]{brocic-cant-arXiv-2026} then shows $\mathit{PSS}(1)\in \mathit{SH}^{eq}(T^{*}L)$ does not vanish, using the surjectivity of $\mathit{PSS}:\mathit{HM}(T^{*}L,\Q)\to \mathit{SH}(T^{*}L)$ onto the summand of contractible loops. This proves that the first Gutt-Hutchings capacity $c_{1}(DT^{*}L)$ is infinite for any aspherical manifold $L$, as desired.\hfill$\square$

\bibliographystyle{amsalpha}
\bibliography{../citations}
\end{document}